\documentclass[11pt, a4paper]{article}
\usepackage{babel}
\usepackage{amsmath,amssymb,latexsym,color,enumerate,amsthm}
\usepackage{indentfirst}
\usepackage{authblk}
\usepackage[mathscr]{eucal}
\usepackage[colorlinks,
            linkcolor=blue,
            anchorcolor=green,
            citecolor=magenta
           ]{hyperref}
\usepackage{setspace}
\newtheorem{theorem}{Theorem}[section]

\newtheorem{lemma}[theorem]{Lemma}
\newtheorem{cor}[theorem]{Corollary}

\newtheorem{conj}{Conjecture}

\newtheorem{rem}[theorem]{Remark}

\newtheorem{example}{Example}[section]
\newtheorem{defin}[theorem]{Definition}
\newtheorem{pro}[theorem]{Problem}

\usepackage{CJK}

\begin{document}

\renewcommand{\baselinestretch}{1.0}

\title{\bf On degree bounds of $k$-uniform hypergraphs with bounded matching number}
\author[1]{Haixiang Zhang\thanks{E-mail: \texttt{zhang-hx22@mails.tsinghua.edu.cn}}}
\author[2]{Mengyu Cao\thanks{Corresponding author. E-mail: \texttt{myucao@ruc.edu.cn}}}
\author[1]{Mei Lu\thanks{E-mail: \texttt{lumei@tsinghua.edu.cn}}}

\affil[1]{\small Department of Mathematical Sciences, Tsinghua University, Beijing 100084, China}
\affil[2]{\small Institute for Mathematical Sciences, Renmin University of China, Beijing 100086, China}
\date{}
\maketitle
\begin{abstract}
We study the connection between the degree sequence of a $k$-uniform hypergraph and the size of its largest matching. Let $\mathcal{F}$ be a $k$-uniform hypergraph on $n$ vertices and let $d_1(\mathcal{F})\ge d_2(\mathcal{F})\ge\cdots\ge d_n(\mathcal{F})$ be the vertex degrees arranged in non-increasing order. For integers $k\ge 2$, $s\ge 2$ and $n > 2sk$, we prove that if the $2sk$-th largest degree satisfies $d_{2sk} > \binom{n-1}{k-1} - \binom{n-s}{k-1},$ then $\mathcal{F}$ contains a matching of size at least $s$. This can be viewed as a generalization of theorems in \cite{GuoLuJiang2023} and \cite{huang2026ell}. Moreover, by relaxing the range of $n$, we obtain the same bound for the $(k+2s-2)$-th largest degree vertex. Note that the number $k+2s-2$ is optimal.

For a $k$-set of vertices $S \subseteq [n]$, the degree of $S$ is defined as $\deg(S) = \sum_{v \in S} \deg(v)$, and the minimum of $\deg(S)$ over all non-edge $k$-subsets $S \notin E(\mathcal{F})$ of $V(\mathcal{F})$ is the \textit{Ore-degree} of $\mathcal{F}$, denoted by $\sigma_k(\mathcal{F})$. Balogh, Palmer and Raeisi \cite{balogh2026} proved: for $s \ge 2$ and $n \ge 3k^2(s-1)$, if $\sigma_k(\mathcal{F}) > k\left(\binom{n-1}{k-1} - \binom{n-s}{k-1}\right),$ then $\mathcal{F}$ contains a matching of size $s$. They also conjectured that the result holds when $n > sk$. As a corollary, we prove that the bound on $n$ can be taken to be linear in $sk$ ($ n \geq 3sk $).
    \medskip


\end{abstract}
\section{Introduction}
For a $k$-uniform hypergraph $\mathcal{F}$ on vertex set $[n]=\{1,\dots,n\}$, we denote by $d_i(\mathcal{F})$ the degree of vertex $i$, i.e. the number of edges containing $i$. Throughout we order the degrees so that $d_1(\mathcal{F})\ge d_2(\mathcal{F})\ge\cdots\ge d_n(\mathcal{F})$. The \emph{matching number} $\nu(\mathcal{F})$ is the size of a largest collection of pairwise disjoint edges in $\mathcal{F}$.

Let $n$ and $k$ be positive integers with $n \geq k$. Denote by $\binom{[n]}{k}$ the family of all $k$-subsets of $[n].$ For any positive integer $t$, a family $\mathcal{F}\subseteq \binom{[n]}{k}$ is said to be \emph{$t$-intersecting} if $|A \cap B|\geq t$ for all $A, B\in\mathcal{F}.$  For convenience, we will simply refer to a $1$-intersecting family as an \emph{intersecting family}.
A fundamental result in extremal hypergraph theory is the Erd\H{o}s-Ko-Rado theorem (EKR theorem for short) \cite{erdos1961intersection}, which characterises the largest intersecting families of $k$-subsets. More precisely, it states that for integers $n \ge k \ge 1$, if $n \ge 2k$ then any family $\mathcal{F}$ of $k$-element subsets of an $n$-element set with the property that $A \cap B \neq \emptyset$ for all $A,B \in \mathcal{F}$ satisfies $|\mathcal{F}| \le \binom{n-1}{k-1}$, and equality holds if and only if $\mathcal{F}$ consists of all $k$-sets containing a fixed element (when $n > 2k$); for $n = 2k$, the family of all $k$-sets complementary to a fixed $k$-set also attains the bound.

The EKR theorem is widely regarded as the cornerstone of extremal set  theory and has many interesting applications and generalizations. For such families, Huang and Zhao \cite{huang2017degree} proved a minimum degree version as follows.

\begin{theorem}{\rm \cite{huang2017degree}}\label{huang1}
Suppose $n \geq 2k+1$ and $\mathcal{F} \subset \binom{[n]}{k}$ is an intersecting family. Then
$$
d_n(\mathcal{F}) \le \binom{n-2}{k-2}.
$$
\end{theorem}

Frankl and Wang \cite{frankl2025largest} conjectured that the same bound actually controls the $(2k+1)$-th largest degree, for any $n\ge 2k+1$,
i.e., $d_{2k+1}(\mathcal{F})\le\binom{n-2}{k-2}$, and they proved that for $n>6k-9$. The conjecture was confirmed by Huang and Rao \cite{huang2026ell}.

\begin{theorem}{\rm \cite{huang2026ell}}\label{huang2}
Suppose $n \geq 2k+1$, $\mathcal{F} \subset \binom{[n]}{k}$ is an intersecting family and degree sequence $d_1(\mathcal{F})\ge d_2(\mathcal{F})\ge\cdots\ge d_n(\mathcal{F})$.  Then
$$
d_{2k+1}(\mathcal{F}) \le \binom{n-2}{k-2}.
$$
\end{theorem}
\begin{theorem}{\rm \cite{frankl2025largest}}\label{s=2}\label{frankl1}
Suppose $n \geq 6k^2$, $\mathcal{F} \subset \binom{[n]}{k}$ is an intersecting family and degree sequence $d_1(\mathcal{F})\ge d_2(\mathcal{F})\ge\cdots\ge d_n(\mathcal{F})$. Then
$$
d_{k+2}(\mathcal{F}) \le \binom{n-2}{k-2}.
$$
\end{theorem}
Huang and Rao \cite{huang2026ell} proved a linear bound $n>\frac{11}{2}k$ for $k>50$.

A natural generalization of the EKR theorem is to relax the intersecting condition (which corresponds to a matching number of $1$) to families with bounded matching number $\nu(\mathcal{F}) \le s$. This leads to the Erd\H{o}s matching conjecture, which predicts that the largest $k$-uniform family on $[n]$ with matching number at most $s$. The conjecture thus extends the classical EKR theorem from the case $s=1$ to arbitrary $s$ and remains one of the central open problems in extremal set theory.

\begin{conj}[Erd\H{o}s matching conjecture {\rm \cite{erdos1965problem}}]
	Let $\mathcal{F}\subseteq \binom{[n]}{k}$ with $\nu(\mathcal{F})=s$, $n\geq sk+k-1$. Then $|\mathcal{F}|\leq \max\{|\mathcal{A}_1|,|\mathcal{A}_k|\} = \max\left\{\binom{n}{k}-\binom{n-s}{k}, \binom{sk+k-1}{k}\right\}$, where $\mathcal{A}_i=\{F\in \binom{[n]}{k}:|F\cap [is+i-1]|\geq i\}$ for $1\leq i \leq k$ and $\nu(\mathcal{A}_i)=s$.
\end{conj}

When $n \geq (s+1)(k+1)$, we can easily show that $|\mathcal{A}_1| > |\mathcal{A}_k|$. Thus, part of the conjecture involves proving that $|\mathcal{F}| \leq |\mathcal{A}_1|$ when $n > n_0$. The bounds for $n_0$ have been progressively improved. Specifically, Bollob\'{a}s, Daykin, and Erd\H{o}s \cite{bollobas1976sets} initially established a bound, which was subsequently refined by Huang, Loh, and Sudakov \cite{huang2012size} to $n_0 = 3k^2s$, and further improved by Frankl, {\L}uczak, and Mieczkowska \cite{frankl2012matchings} to $n_0 = 3k^2s/(2\log k)$. Later, Frankl \cite{frankl2013improved} achieved $n_0 = (2s+1)k - s$, and Frankl and Kupavskii \cite{frankl2022} ultimately reduced it to $n_0 = \frac{5}{3}sk - \frac{2}{3}s$.

Bollobás, Daykin, and Erd\H{o}s \cite{bollobas1976sets} also investigated a minimum degree version of the conjecture. They showed that if $ n > 2k^3(s-1) $ and $ \mathcal{F} $ is an $ k $-uniform hypergraph on $ n $ vertices with
$
\delta(\mathcal{F}) \geq \binom{n-1}{k-1} - \binom{n-s}{k-1},
$
then $ \mathcal{F} $ has matching number at least $ s $. This result was later improved by Huang and Zhao \cite{huang2017degree}, who reduced the bound on $ n $ to $ n \geq 3k^2 s $, and further refined by Guo, Lu and Jiang \cite{GuoLuJiang2023} as follows.

\begin{theorem}{\rm\cite{huang2017degree}}\label{degree mat}
	Let $n,k,s$ be an integers such that $n\geq 3k^2 s$. If $\mathcal{F}\subseteq\binom{[n]}{k}$ with
	$
	d_n(\mathcal{F}) > \binom{n-1}{k-1} - \binom{n-s}{k-1},
	$
	then $\nu(\mathcal{F}) \ge s$.
\end{theorem}

\begin{theorem}{\rm\cite{GuoLuJiang2023}}\label{degree mat2}
	Let $k \ge 3$ be an integer. There exists a constant $n_0=n_0(k)$ such that the following holds. Let $n, s$ be integers such that $n \ge \max\{n_0,2sk\}$. If $\mathcal{F}\subseteq\binom{[n]}{k}$ with
	$
	d_n(\mathcal{F}) > \binom{n-1}{k-1} - \binom{n-s}{k-1},
	$
	then $\nu(\mathcal{F}) \ge s$.
\end{theorem}
We generalize Theorem \ref{huang2} from intersecting families to families with matching number less than $s$. Notice that the bound on $n$ is the same as that in Theorem \ref{degree mat2}. Hence, this result can also be viewed as a generalization of Theorem \ref{degree mat2}.
\begin{theorem}\label{thm:main}
Let $k\ge 2$ and $s\ge 1$ be integers.  There exists a constant $n_0=n_0(k)$ such that the following holds. Let $\mathcal{F}$ be a $k$-uniform hypergraph on $n$ vertices with $\nu(\mathcal{F})< s$ and degree sequence $d_1(\mathcal{F})\ge d_2(\mathcal{F})\ge\cdots\ge d_n(\mathcal{F})$.
If $n>2sk\geq n_0$, then
$$
d_{2sk}(\mathcal{F}) \leq \binom{n-1}{k-1}-\binom{n-s}{k-1}.
$$
\end{theorem}
When $sk$ is small, $d_{2sk}$ can be replaced by $d_{3k^2 s}$.
\begin{rem}\label{rem}
	Let $k\ge 2$ and $s\ge 1$ be integers.  There exists a constant $n_0=n_0(k)$ such that the following holds. Let $\mathcal{F}$ be a $k$-uniform hypergraph on $n$ vertices with $\nu(\mathcal{F})< s$ and degree sequence $d_1(\mathcal{F})\ge d_2(\mathcal{F})\ge\cdots\ge d_n(\mathcal{F})$. If $n>3k^2 s$, then
	$$
	d_{3k^2 s}(\mathcal{F}) \leq \binom{n-1}{k-1}-\binom{n-s}{k-1}.
	$$
\end{rem}

Let $ \mathcal{F} \subseteq \binom{[n]}{k} $ be an $ k$-uniform hypergraph on vertex set $[n] = \{1, 2, \ldots, n\}$. For an $ k $-set of vertices $ S \subseteq [n] $, the degree of $ S $ is defined as $ \deg(S) = \sum_{v \in S} \deg(v) $ and the minimum of $ \deg(S) $ over all non-edge-$ k $-subsets $ S \notin E(\mathcal{F}) $ of $ V(\mathcal{F}) $ is the \textit{Ore-degree} of $ \mathcal{F} $, denoted by $ \sigma_k(\mathcal{F}) $. Balogh, Palmer and Raeisi proved the following result.
\begin{theorem}{\rm \cite{balogh2026}}
	Let $ s \geq 2 $ and $ n \geq 3k^2(s-1) $. If $ \mathcal{F} $ is an $ k $-uniform hypergraph on $ n $ vertices and
	$$
	\sigma_k(\mathcal{F}) > k \left( \binom{n-1}{k-1} - \binom{n-s}{k-1} \right),
	$$
	then $ \mathcal{F} $ contains $ s $ pairwise disjoint edges.
\end{theorem}
And they conjectured the result holds when $n>sk$. For $ k = 2 $, a generalization of this conjecture was proved in \cite{OMIDI2024113785}, but it remains open for $ k \geq 3 $. We prove that the bound on $n$ can be linear in $sk$ for large $s$.

\begin{cor} \label{cor}
	Let $k\ge 2$ and $s\ge 1$ be integers.  There exists a constant $n_0=n_0(k)$ such that the following holds.  If $ \mathcal{F} $ is a $k$-uniform hypergraph on $n$ vertices with $n\geq 3sk$ and $2sk\geq n_0(k)$, and
	$$
	\sigma_k(\mathcal{F}) > k \left( \binom{n-1}{k-1} - \binom{n-s}{k-1} \right),
	$$
	then $ \mathcal{F} $ contains $ s $ pairwise disjoint edges.
\end{cor}

By sacrificing the range of $n$ in Theorem \ref{thm:main}, we can obtain the same bound for the $(k+2s-2)$-th largest degree vertex.
\begin{theorem}\label{thm:main2}
Let $k\ge 2$, $s\ge 2$ be integers and $\mathcal{F}$ be a $k$-uniform hypergraph on $n$ vertices with $\nu(\mathcal{F})< s$ and degree sequence $d_1(\mathcal{F})\ge d_2(\mathcal{F})\ge\cdots\ge d_n(\mathcal{F})$. There exists an integer $n_0=n_0(k,s)$ such that if $n\geq n_0$, then
$$
d_{k+2s-2}(\mathcal{F}) \leq \binom{n-1}{k-1}-\binom{n-s}{k-1}.
$$
\end{theorem}
\begin{rem} We need to point out that $k+2s-2$ is {\bf optimal} here.
	Define
\begin{align*}
     \mathcal{G}_1&=\left\{G\in\binom{[n]}{k}:G\cap[s-1]\neq\emptyset,G\cap[s,k+2s-3]\neq\emptyset\right\},\\
     \mathcal{G}_2&=\left\{G\in\binom{[s,k+2s-3]}{k}:s\in G\right\}.
\end{align*}
Every member of $\mathcal{G}_1$ meets $[s-1]$, so $\nu(\mathcal{G}_1)\leq s-1$, while $\nu(\mathcal{G}_2)=1$.
Consider the family
$
    \mathcal{G}=\mathcal{G}_1\cup\mathcal{G}_2.
$
We claim that $\nu(\mathcal{G})<s$.
If there exists a matching of size $s$ in $\mathcal{G}$, it must choose one set from $\mathcal{G}_2$ and $s-1$ sets from $\mathcal{G}_1$. However, if we fix a set $G\in \mathcal{G}_2$, the $s-1$ sets must be chosen from
$$
\left\{F\in\binom{[n]}{k}:F\cap[s-1]\neq\emptyset,\ F\cap([s,k+2s-3]\setminus G)\neq\emptyset\right\}.
$$
The set $[s,k+2s-3]\setminus G$ is a vertex cover of size $s-2$ of the above family. Thus, we can find at most $s-2$ pairwise disjoint members of $\mathcal{G}_1$ avoiding $G$, a contradiction. Hence $\nu(\mathcal{G})<s$. For sufficiently large $n$, one can choose $s-1$ pairwise disjoint members of $\mathcal{G}_1$, and therefore $\nu(\mathcal{G})=s-1$.

For $i\in[s-1]$, we have
$$
d_i(\mathcal{G})\geq \binom{n-1}{k-1}-\binom{n-k-s+1}{k-1}
>\binom{n-1}{k-1}-\binom{n-s}{k-1}.
$$
For $i\in[s,k+2s-3]$, the contribution of $\mathcal{G}_1$ is exactly $\binom{n-1}{k-1}-\binom{n-s}{k-1}$, and $i$ has positive degree in $\mathcal{G}_2$. Hence $d_i(\mathcal{G}) > \binom{n-1}{k-1}-\binom{n-s}{k-1}$ for every $i \in [k+2s-3]$.

Indeed, there are many other families that play the same role as $\mathcal{G}$. For instance, we may replace $\mathcal{G}_2$ by any intersecting family in $\binom{[s, k+2s-3]}{k}$ that has positive minimum degree. For example, when $s \geq 3$, the following two families also preserve the property that the matching number is less than $s$ and that the $(k+2s-3)$-th largest degree exceeds $\binom{n-1}{k-1} - \binom{n-s}{k-1}$:
\begin{align*}
    \mathcal{H}_1=&\left\{H\in\binom{[n]}{k}:H\cap[s-1]\neq\emptyset,H\cap[s,k+2s-3]\neq\emptyset\right\}\\ \cup&\left\{H\in\binom{[s,k+2s-3]}{k}:s\in H,H\cap[s+1,k+s]\neq\emptyset\right\}\cup\{[s+1,k+s]\};\\
    \mathcal{H}_2=&\left\{H\in\binom{[n]}{k}:H\cap[s-1]\neq\emptyset,H\cap[s,k+2s-3]\neq\emptyset\right\}\\ \cup&\left\{H\in\binom{[s,k+2s-3]}{k}:|H\cap[s,s+2]|\geq 3\right\}.
\end{align*}
\end{rem}

The paper is organised as follows. Section~\ref{shift} recalls the necessary shifting lemmas and the Kruskal--Katona theorem. Section~\ref{2sk} contains the proof of Theorem~\ref{thm:main} and Corollary~\ref{cor}. Section~\ref{large n} proves the bound for the $(k+2s-2)$-th largest degree by introducing some results on $\Delta$-systems and structural results for simple graphs. Section~\ref{open} discusses extensions and an open problem.

\section{Shifting and shadow}\label{shift}
The \emph{shadow} of $\mathcal{F}\subseteq\binom{[n]}{k}$, denoted by $\partial\mathcal{F}$, is the set of all $(k-1)$-element subsets that are contained in at least one member of $\mathcal{F}$. The \emph{colexicographic order} on $k$-sets compares two sets by their largest element, then the second largest, etc.; equivalently, $A < B$ if $\max(A \setminus B) < \max(B \setminus A)$. The initial segment of this order of a given size forms the family that, under the Kruskal-Katona theorem, minimizes the size of the shadow for a given number of $k$-sets.

\begin{theorem}[Kruskal-Katona Theorem]
For any $\mathcal{F}\subseteq\binom{[n]}{k}$ with size $m$, we have
$$|\partial\mathcal{F}| \geq |\partial\mathcal{C}(n,k,m)|,$$
where $\mathcal{C}(n,k,m)$ denotes the family consisting of the first $m$ $k$-subsets of $[n]$ in colexicographic order.
\end{theorem}

\emph{Shifting} (also known as compression) is a useful combinatorial tool. Given a $k$-uniform family $\mathcal{F} \subseteq \binom{[n]}{k}$ and two indices $i, j$ with $i < j$, the shift operation $\mathcal{S}_{i,j}$ transforms $\mathcal{F}$ by replacing every set $A \in \mathcal{F}$ with
$$
\mathcal{S}_{i,j}(A) =
\begin{cases}
(A \setminus \{j\}) \cup \{i\}, & \text{if } i \notin A,\; j \in A,\; \text{and } (A \setminus \{j\}) \cup \{i\} \notin \mathcal{F},\\
A, & \text{otherwise}.
\end{cases}
$$
Applying $\mathcal{S}_{i,j}$ to every member of $\mathcal{F}$ yields a new family $\mathcal{S}_{i,j}(\mathcal{F})$.

In \cite{huang2026ell}, Huang and Rao introduced the concept of $\ell$-shifted families. A family $\mathcal{F}$ is called \emph{$\ell$-shifted} if it is invariant under all shifts with $i \in [\ell]$ and $j \notin [\ell]$; that is, $\mathcal{S}_{i,j}(\mathcal{F}) = \mathcal{F}$ for every pair $(i,j)$ satisfying $i \in [\ell],\; j \notin [\ell]$.

Here are some basic properties for  an $\ell$-shifted family.
\begin{lemma}\label{ell-shifted}
Let $\mathcal{F}\subset \binom{[n]}{k}$ be an $\ell$-shifted $k$-uniform family with matching number $\nu(\mathcal{F})<s$. Define the $($not necessarily uniform$)$ family $\mathcal{G}$ on $[\ell]$ by
$$
\mathcal{G} = \{A \cap [\ell] \mid A \in \mathcal{F}\}.
$$
Then the following statements hold.
\begin{enumerate}[{\rm (1)}]
    \item $\mathcal{G}$ is upward closed among subsets of $[\ell]$ of size at most $k$, meaning that for all $A \subset B \subset [\ell]$, if $A \in \mathcal{G}$ and $|B| \leq k$, then $B \in \mathcal{G}$.
    \item If $\ell\geq sk$, then $\nu(\mathcal{G})<s$.
\end{enumerate}
\end{lemma}
\begin{proof}
(1) Suppose $ G \in \mathcal{G} $ and $|G| < k$. Then there exists $ F \in \mathcal{F} $ such that $ F \cap [\ell] = G $.
For any $ j \in [\ell] \setminus G $, pick an element $ i \in F \setminus G $ (note that $ i \notin [\ell] $ because $ F \cap [\ell] = G $). By invariance under the shift $\mathcal{S}_{j,i}$, we obtain

$$
F' := (F \cup \{j\}) \setminus \{i\} \in \mathcal{F}.
$$

Consequently, $ F' \cap [\ell] = (F \cup \{j\}) \cap [\ell] = G \cup \{j\} $. Hence $ G \cup \{j\} \in \mathcal{G} $ for every $ j \in [\ell] \setminus G $, which shows that $ \mathcal{G} $ is upward closed with respect to inclusion for all sets of size strictly less than $ k $.

(2) Suppose $\{G_1, G_2,\ldots,G_s\} \subseteq \mathcal{G}$ is a matching of size $s$. Since $\ell \ge sk$, we have
$$
\left|[\ell]\setminus \bigcup_{i=2}^{s} G_i\right| \ge sk - (s-1)k = k.
$$
Thus we can choose $F_1 \in \binom{[\ell]}{k}$ such that $G_1 \subseteq F_1$ and $F_1\cap (\cup_{i=2}^{s} G_i)=\emptyset$. By (1), $F_1 \in \mathcal{G}$, and hence $F_1 \in \mathcal{F}$.
Assume we have obtained $\{F_1,F_2,\ldots,F_m\}\subseteq \mathcal{F}$ which is a matching of size $m$, where $1 \le m \le s-1$. Since $\ell \ge sk$, we have
$$
\left|[\ell]\setminus \left(\bigcup_{i=1}^{m} F_i \cup \bigcup_{i=m+2}^{s} G_i\right)\right| \ge sk - (s-1)k = k.
$$
Thus, we can choose $F_{m+1} \in \binom{[\ell]}{k}$ such that $G_{m+1} \subseteq F_{m+1}$ and $F_{m+1}\cap \left(\bigcup_{i=1}^{m} F_i \cup \bigcup_{i=m+2}^{s} G_i\right)=\emptyset$. By (1), $F_{m+1} \in \mathcal{G}$, and hence $F_{m+1} \in \mathcal{F}$.

Finally, we obtain $\{F_1,F_2,\ldots,F_s\}\subseteq \mathcal{F}$ which is a matching of size $s$, a contradiction with $\nu(\mathcal{F}) < s$.
\end{proof}

 Let $\mathcal{G}$ be a (not necessarily uniform) family of subsets of $[n]$. Define $$\mathcal{G}_i^{(m)}=\{A\setminus\{i\}: i \in A\in \mathcal{G}, \, |A| = m\}$$
 For $i \in [n]$, the $m$-degree of $i$ in $\mathcal{G}$, denoted by $d_i^{(m)}(\mathcal{G})$, is the number of sets in $\mathcal{G}$ of size $m$ that contain $i$, that is
$$
d_i^{(m)}(\mathcal{G}) := |\mathcal{G}_i^{(m)}|.
$$
\begin{lemma}\label{main:lem}
Let $k\ge 2$ and $s\ge 1$ be integers. There exists a constant $n_0=n_0(k)$ such that the following holds. Assume either $n \geq \max\{2sk,n_0\}$ or $n\geq 3k^2 s$, and let $\mathcal{G}\subseteq 2^{[n]}$ be a {\rm(}not necessarily uniform{\rm)} family with $\emptyset\notin\mathcal{G}$ and $\nu(\mathcal{G})<s$. Then there exists an element $i \in [n]$ such that for every $m \leq k$,
$$
d_i^{(m)}(\mathcal{G}) \leq \binom{n-1}{m-1}-\binom{n-s}{m-1}.
$$
\end{lemma}
\begin{proof}
Let
$$
\mathcal{F}= \{F\subseteq[n]:G\subseteq F\text{ for some }G\in\mathcal{G}\}
$$
be the upward closure of $\mathcal{G}$. Since $\emptyset\notin\mathcal{G}$, any matching in $\mathcal{F}$ contains pairwise disjoint nonempty members of $\mathcal{G}$, one contained in each of its members. Hence $\nu(\mathcal{F})<s$. Moreover, $d_i^{(m)}(\mathcal{F})\geq d_i^{(m)}(\mathcal{G})$ for all $1\leq m\leq k$ and $1\leq i\leq n$. Thus, it suffices to show that there exists an element $i \in [n]$ such that for every $m \leq k$,
$$
d_i^{(m)}(\mathcal{F}) \leq \binom{n-1}{m-1}-\binom{n-s}{m-1}.
$$
We proceed by downward induction on $m$.

\noindent\textbf{Induction base:} $m=k$. If $k=2$, the conclusion follows from the elementary graph fact in Lemma~\ref{str1}. If $k\geq3$, apply Theorem~\ref{degree mat2} under the first alternative on $n$, and Theorem~\ref{degree mat} under the second alternative, to the $k$-uniform layer $\mathcal{F}^{(k)}=\mathcal{F}\cap\binom{[n]}{k}$. Since $\nu(\mathcal{F}^{(k)})<s$, in either case there exists $i_0\in[n]$ such that
$$
d_{i_0}^{(k)}(\mathcal{F}) \leq \binom{n-1}{k-1}-\binom{n-s}{k-1}.
$$

\noindent\textbf{Induction step:} Assume that the inequality holds for some $m$ with $2\leq m\leq k$, i.e.
$$
d_{i_0}^{(m)}(\mathcal{F}) \leq \binom{n-1}{m-1}-\binom{n-s}{m-1}.
$$
We show that it then holds for $m-1$. Suppose, to the contrary, that
$$
d_{i_0}^{(m-1)}(\mathcal{F}) \geq \binom{n-1}{m-2}-\binom{n-s}{m-2}+1.
$$
Put $V=[n]\setminus\{i_0\}$ and
$$
\mathcal{A}=\mathcal{F}_{i_0}^{(m-1)}\subseteq\binom{V}{m-2},
\qquad
\mathcal{A}^{c}=\{V\setminus A:A\in\mathcal{A}\}\subseteq\binom{V}{n-m+1},
$$
where complements are taken relative to $V$. If $D\in\partial\mathcal{A}^{c}$, choose $C\in\mathcal{A}^{c}$ with $D\subset C$. Then $V\setminus C\in\mathcal{A}$ and $V\setminus C\subset V\setminus D$. Since $\mathcal{F}$ is upward closed, $V\setminus D\in\mathcal{F}_{i_0}^{(m)}$. Consequently,
$$
d_{i_0}^{(m)}(\mathcal{F})\geq |\partial\mathcal{A}^{c}|.
$$

Relabel $V$ as $[n-1]$, set $r=n-m+1$, and let $T=[n-s+1,n-1]$, so $|T|=s-1$. The number of $r$-subsets of $[n-1]$ that do not contain $T$ is
$$
M=\binom{n-1}{m-2}-\binom{n-s}{m-2}.
$$
These sets form the first $M$ members of the colexicographic order: if an $r$-set $A$ misses an element of $T$ and an $r$-set $B$ contains $T$, then the largest element of $A\triangle B$ belongs to $B$. Among the $r$-sets containing $T$, the first one in colexicographic order is
$$
Q_0=[n-m-s+2]\cup T.
$$
Let $\mathcal{C}$ be this colex initial segment of length $M+1$. For $m\geq3$, every $(r-1)$-set not containing $T$ belongs to $\partial\mathcal{C}$: if such a set $D$ misses $t\in T$, then $|[n-1]\setminus D|=m-1\geq2$, so one can choose $x\in[n-1]\setminus(D\cup\{t\})$, and $D\cup\{x\}$ is one of the first $M$ members of $\mathcal{C}$. For $m=2$, we have $M=0$ and $Q_0=[n-1]$, whose shadow contains all $(r-1)$-sets. In addition,
$$
Q_0\setminus\{n-m-s+2\}=[n-m-s+1]\cup T
$$
belongs to $\partial\mathcal{C}$ and contains $T$. Hence
\begin{align*}
|\partial\mathcal{C}|
&\geq \binom{n-1}{r-1}-\binom{n-s}{r-s}+1\\
&=\binom{n-1}{m-1}-\binom{n-s}{m-1}+1.
\end{align*}
Since $|\mathcal{A}^{c}|\geq M+1$, the Kruskal--Katona theorem and the nesting of colex initial segments give
$$
|\partial\mathcal{A}^{c}|\geq |\partial\mathcal{C}|.
$$
Therefore,
$$
d_{i_0}^{(m)}(\mathcal{F})\geq \binom{n-1}{m-1}-\binom{n-s}{m-1}+1,
$$
contradicting the induction hypothesis. Thus,
$$
d_{i_0}^{(m-1)}(\mathcal{F}) \leq \binom{n-1}{m-2}-\binom{n-s}{m-2}.
$$
This completes the downward induction.
\end{proof}
\section{The $2sk$-th largest degree}\label{2sk}
Now, we are ready to prove Theorem \ref{thm:main} and Corollary \ref{cor}.

\noindent\textbf{Proof of Theorem \ref{thm:main}}
If $s=1$, then $\mathcal{F}=\emptyset$ and the conclusion is immediate. Hence assume $s\geq2$. Let $\mathcal{F}_0$ denote the original family, with the vertices labelled so that
$d_1(\mathcal{F}_0)\geq\cdots\geq d_n(\mathcal{F}_0)$.

Apply shifts $\mathcal{S}_{i,j}$ with $1\leq i\leq 2sk<j\leq n$ until no such shift changes the family, and denote the resulting family by $\mathcal{F}$. A shift does not increase the matching number. For such a shift, the degree of its target vertex in $[2sk]$ does not decrease, and every other vertex of $[2sk]$ remains in each shifted edge in which it originally occurred. Hence the degree of every vertex in $[2sk]$ does not decrease during this procedure. Thus
$$
\nu(\mathcal{F})<s
\qquad\text{and}\qquad
d_i(\mathcal{F})\geq d_i(\mathcal{F}_0)\quad(i\in[2sk]).
$$
The family $\mathcal{F}$ is $2sk$-shifted.

Define $\mathcal{G}=\{F\cap[2sk]:F\in\mathcal{F}\}$. By Lemma~\ref{ell-shifted}, $\mathcal{G}$ is an upward closed family on $[2sk]$, every member has size at most $k$, and $\nu(\mathcal{G})<s$. In particular, $\emptyset\notin\mathcal{G}$, since otherwise upward closure would put every $k$-subset of $[2sk]$ in $\mathcal{G}$; each such trace is itself a member of $\mathcal{F}$, giving $s$ pairwise disjoint members of $\mathcal{F}$. Applying Lemma~\ref{main:lem} with $n=2sk$, we obtain $i\in[2sk]$ such that for every $m\leq k$,
$$
d_i^{(m)}(\mathcal{G})\leq \binom{2sk-1}{m-1}-\binom{2sk-s}{m-1}.
$$

For a fixed $A\in\mathcal{G}^{(m)}_i$, there are at most $\binom{n-2sk}{k-m}$ sets $B\in\mathcal{F}$ satisfying $B\cap[2sk]=A$. Consequently,
\begin{align*}
d_i(\mathcal{F})
&\leq \sum_{m=1}^k d_i^{(m)}(\mathcal{G})\binom{n-2sk}{k-m}\\
&\leq \sum_{m=1}^k \left(\binom{2sk-1}{m-1}-\binom{2sk-s}{m-1}\right)\binom{n-2sk}{k-m}\\
&=\binom{n-1}{k-1}-\binom{n-s}{k-1},
\end{align*}
where the last equality is Vandermonde's identity. Since $i\in[2sk]$,
$$
d_{2sk}(\mathcal{F}_0)\leq d_i(\mathcal{F}_0)\leq d_i(\mathcal{F})
\leq\binom{n-1}{k-1}-\binom{n-s}{k-1}.
$$
This completes the proof.

\noindent\textbf{Proof of Remark \ref{rem}}
If $s=1$, the conclusion is immediate. Assume $s\geq2$, and let $\mathcal{F}_0$ be the original family with non-increasing degree sequence. Apply all shifts $\mathcal{S}_{i,j}$ with $1\leq i\leq 3k^2s<j\leq n$, and denote the resulting $3k^2s$-shifted family by $\mathcal{F}$. As in the proof of Theorem~\ref{thm:main},
$$
\nu(\mathcal{F})<s
\qquad\text{and}\qquad
d_i(\mathcal{F})\geq d_i(\mathcal{F}_0)\quad(i\in[3k^2s]).
$$
Define $\mathcal{G}=\{F\cap[3k^2s]:F\in\mathcal{F}\}$. By Lemma~\ref{ell-shifted}, $\mathcal{G}$ is upward closed and has matching number less than $s$; by the same argument as above, it does not contain the empty set. Applying Lemma~\ref{main:lem} with $n=3k^2s$, we obtain $i\in[3k^2s]$ such that for every $m\leq k$,
$$
d_i^{(m)}(\mathcal{G})\leq \binom{3k^2s-1}{m-1}-\binom{3k^2s-s}{m-1}.
$$
Thus,
\begin{align*}
d_i(\mathcal{F})
&\leq \sum_{m=1}^k d_i^{(m)}(\mathcal{G})\binom{n-3k^2s}{k-m}\\
&\leq \sum_{m=1}^k \left(\binom{3k^2s-1}{m-1}-\binom{3k^2s-s}{m-1}\right)\binom{n-3k^2s}{k-m}\\
&=\binom{n-1}{k-1}-\binom{n-s}{k-1}.
\end{align*}
Since $i\in[3k^2s]$, comparison with the original degrees gives the required bound on $d_{3k^2s}(\mathcal{F}_0)$.
\hfill\qed

\noindent\textbf{Proof of Corollary \ref{cor}}
Assume $ \mathcal{F} \subseteq \binom{[n]}{k} $ is a family with matching number $ \nu(\mathcal{F}) < s $. It suffices to show that there exists a set $ F \notin \mathcal{F} $ such that $ \deg(F) \leq k \left( \binom{n-1}{k-1} - \binom{n-s}{k-1} \right) $.

By Theorem \ref{thm:main}, we have $ d_i(\mathcal{F}) \leq \binom{n-1}{k-1} - \binom{n-s}{k-1} $ for any $ i \in [2sk, n] $. Note that $ \nu\left( \binom{[2sk,n]}{k} \right) \geq s $, since $ n \geq 3sk $. Hence there exists some $ F \in \binom{[2sk,n]}{k} \setminus \mathcal{F} $; otherwise $ \nu(\mathcal{F}) \geq \nu\left( \binom{[2sk,n]}{k} \right) \geq s $, a contradiction. For such an $ F $, the inequality $$ \deg(F) \leq k \left( \binom{n-1}{k-1} - \binom{n-s}{k-1} \right) $$ implies that $ \sigma_k(\mathcal{F}) \leq k \left( \binom{n-1}{k-1} - \binom{n-s}{k-1} \right)  $, a contradiction.
\hfill\qed

\section{Tight bound for large $n$}\label{large n}
\subsection{Sunflower base}
The \emph{$\Delta$-system method} (also known as the \emph{sunflower method}) was introduced by Erd\H{o}s and Rado in 1960 \cite{erdos1960intersection}. They proved a Ramsey-type result stating that any sufficiently large family of sets of bounded size contains a homogeneous substructure called a \emph{$\Delta$-system} or \emph{sunflower}, i.e., a collection of sets in which every two intersect in the same common set (the kernel). This result has become a fundamental tool in extremal set theory.

The core idea of the method is to construct a structured \emph{base} that covers most of the original family, allowing one to bound the size of the family using the properties of the base. Deza, Erd\H{o}s and Frankl \cite{DEF78} first applied this method systematically to study families with prescribed intersection sizes. Subsequently, Frankl and F\H{u}redi \cite{FF85, FF87} developed the method further, solving the ``forbidden one intersection'' problem and various Tur\'an-type problems. The $\Delta$-system method has found applications not only in extremal set theory but also in computational complexity (e.g., the works of Razborov and Alon--Boppana \cite{AB87}) and in the study of hypergraph Tur\'an problems.

\begin{defin}[$\Delta$-system]
$\{F_1,F_2,\ldots,F_m\}\subseteq \binom{[n]}{k}$ is called a $\Delta$-system of cardinality $m$ with kernel $K$, if $K=F_i\cap F_j$ for any $1\leq i< j\leq m$. We also call a $\Delta$-system of size $m$ as an $m$-sunflower.
\end{defin}

\begin{defin}[$\Delta$-base] Set
	$\mathcal{B}_{s,k}(\mathcal{F})=\mathcal{F}\cup \{E\in 2^{[n]}:|E|>0,$ there exists a $\Delta$-system in $\mathcal{F}$ with kernel $E$ of cardinality $(sk-k+1)^{|E|}\}$.
	Let $\mathcal{B}_{s,k}^*(\mathcal{F})$ be the set of minimal elements in $\mathcal{B}_{s,k}(\mathcal{F})$ under inclusion. We call $\mathcal{B}_{s,k}^*(\mathcal{F})$ the $\Delta$-base of $\mathcal{F}$. Let
	$\mathcal{B}_{s,k}^{* (i)}(\mathcal{F})=\{F\in \mathcal{B}_{s,k}^*(\mathcal{F}): |F|=i\}$, $i\in [k]$.
\end{defin}

Notice the fact that $\mathcal{F}\subseteq \mathcal{B}_{s,k}(\mathcal{F})$, and then $\mathcal{F}\subseteq \cup_{B\in \mathcal{B}_{s,k}^*(\mathcal{F})}\{F\in\binom{[n]}{k}:B\subseteq F\}$. Here are some inequalities.

\begin{lemma}\label{upp1}
	Let $\mathcal{F}\subseteq\binom{[n]}{k}$ be a $k$-uniform family. Then we have
	\begin{align*}
		(1)~~|\mathcal{F}|\leq& \sum^k_{i=1}|\mathcal{B}_{s,k}^{* (i)}(\mathcal{F})|\binom{n-i}{k-i};\\
		(2)~~d_m(\mathcal{F})\leq& \sum^k_{i=1}|\mathcal{B}_{s,k}^{* (i)}(\mathcal{F},m)|\binom{n-i}{k-i}+\sum^k_{i=1}|\mathcal{B}_{s,k}^{* (i)}(\mathcal{F})\setminus \mathcal{B}_{s,k}^{* (i)}(\mathcal{F},m)|\binom{n-i-1}{k-i-1}
	\end{align*}
	for any $m\in [n]$, where $\mathcal{B}_{s,k}^{* (i)}(\mathcal{F},m)=\{F\in \mathcal{B}_{s,k}^{* (i)}(\mathcal{F}):m\in F\}$.
\end{lemma}
\begin{proof}
	Note that $\mathcal{F}\subseteq \mathcal{B}_{s,k}(\mathcal{F})$, and hence
	$$
	\mathcal{F}\subseteq \bigcup_{B\in \mathcal{B}_{s,k}^*(\mathcal{F})}\left\{F\in\binom{[n]}{k}:B\subseteq F\right\}.
	$$
	Counting all $k$-sets containing each minimal base gives (1). For (2), a base $B$ containing $m$ contributes at most $\binom{n-|B|}{k-|B|}$ sets to the degree of $m$, while a base $B$ not containing $m$ contributes at most $\binom{n-|B|-1}{k-|B|-1}$. Summing these bounds proves (2), with the convention that a binomial coefficient with negative lower index is zero.
\end{proof}

An important property of the $\Delta$-base is that it preserves the matching number of the original set family.
\begin{lemma}\label{mat for 2}
	Let $\mathcal{F}\subseteq\binom{[n]}{k}$ be a family with matching number $\nu(\mathcal{F})<s$. Then, $\nu(\mathcal{B}_{s,k}(\mathcal{F}))=\nu(\mathcal{B}^*_{s,k}(\mathcal{F}))<s$.
\end{lemma}

\begin{proof}
Every member of $\mathcal{B}_{s,k}(\mathcal{F})$ contains a minimal member of $\mathcal{B}_{s,k}(\mathcal{F})$. Therefore, a matching in $\mathcal{B}_{s,k}(\mathcal{F})$ can be replaced, member by member, by a matching of the same size in $\mathcal{B}_{s,k}^*(\mathcal{F})$. Since $\mathcal{B}_{s,k}^*(\mathcal{F})\subseteq\mathcal{B}_{s,k}(\mathcal{F})$, we have
$$
\nu(\mathcal{B}_{s,k}(\mathcal{F}))=\nu(\mathcal{B}_{s,k}^*(\mathcal{F})).
$$

Suppose, for a contradiction, that $B_1,\ldots,B_s\in\mathcal{B}_{s,k}^*(\mathcal{F})$ are pairwise disjoint. Put $p=sk-k+1=k(s-1)+1$. Without loss of generality, assume $B_i\in\mathcal{F}$ for $i\in[t]$ and $B_i\notin\mathcal{F}$ for $i\in[t+1,s]$, where $0\leq t<s$. For each $i\in[t+1,s]$, choose a $p$-sunflower $\mathcal{F}_i\subseteq\mathcal{F}$ with kernel $B_i$; such a sunflower exists because the definition provides at least $p^{|B_i|}\geq p$ petals.

For $i\in[t]$, set $F_i=B_i$. We choose $F_i\in\mathcal{F}_i$ successively for $i=t+1,\ldots,s$ so that all chosen sets are pairwise disjoint and each $F_i$ avoids every kernel $B_j$ with $j\neq i$. Suppose $F_1,\ldots,F_m$ have already been chosen. The forbidden set
$$
\left(\bigcup_{j=1}^{m}F_j\right)\cup\left(\bigcup_{j=m+2}^{s}B_j\right)
$$
contains at most $k(s-1)=p-1$ vertices and is disjoint from $B_{m+1}$. Outside the kernel $B_{m+1}$, the petals of $\mathcal{F}_{m+1}$ are pairwise disjoint, so each forbidden vertex meets at most one petal. Hence at least one of the $p$ petals avoids the entire forbidden set; choose it as $F_{m+1}$.

Continuing in this way produces $s$ pairwise disjoint members $F_1,\ldots,F_s$ of $\mathcal{F}$, contradicting $\nu(\mathcal{F})<s$. Therefore
$$
\nu(\mathcal{B}_{s,k}(\mathcal{F}))=\nu(\mathcal{B}_{s,k}^*(\mathcal{F}))<s.
$$
\end{proof}

Another important property of the $\Delta$-base is that it does not contain a large $\Delta$-system within the same layer.
\begin{lemma}\label{nosunflower}
Let $s\geq2$ and $\mathcal{F}\subseteq\binom{[n]}{k}$ satisfy $\nu(\mathcal{F})<s$. Then $\mathcal{B}_{s,k}^{*(i+1)}(\mathcal{F})$ contains no $\Delta$-system of cardinality $(sk-k+1)^i$, for any $1\leq i\leq k-1$.
\end{lemma}

\begin{proof}
Put $p=sk-k+1$. Suppose, for a contradiction, that
$$
\{B_1,B_2,\ldots,B_{p^i}\}\subseteq\mathcal{B}_{s,k}^{*(i+1)}(\mathcal{F})
$$
forms a $\Delta$-system with kernel $K$. Then $|K|\leq i$. If $K=\emptyset$, the sets $B_1,\ldots,B_{p^i}$ form a matching of size $p^i\geq s$, contradicting Lemma~\ref{mat for 2}. Hence $K\neq\emptyset$.

When $i=k-1$, every $B_j$ belongs to $\mathcal{F}$: since $p^k>1$, a $k$-set cannot be the kernel of a $p^k$-sunflower consisting of distinct $k$-sets. Thus the displayed family is a sunflower in $\mathcal{F}$ with kernel $K$. Since $p^i\geq p^{|K|}$, we have $K\in\mathcal{B}_{s,k}(\mathcal{F})$, contradicting the minimality of the sets $B_j$.

Now suppose $1\leq i\leq k-2$. For each $j\in[p^i]$, the set $B_j$ has size $i+1<k$, and hence there exists a sunflower
$$
\mathcal{F}_j=\{F^j_1,\ldots,F^j_{p^{i+1}}\}\subseteq\mathcal{F}
$$
with kernel $B_j$. We choose one member from each $\mathcal{F}_j$ successively so that the chosen members form a sunflower with kernel $K$ and the member chosen from $\mathcal{F}_j$ avoids $B_r\setminus K$ for every $r\neq j$.

Suppose $F^1_{\ell_1},\ldots,F^m_{\ell_m}$ have been chosen. The union of the sets $F^j_{\ell_j}\setminus K$ for $j\leq m$ and the sets $B_j\setminus K$ for $j\geq m+2$ has fewer than $kp^i$ vertices. Since $p>k$, we have $kp^i<p^{i+1}$. Every forbidden vertex lies outside the kernel $B_{m+1}$ and therefore meets at most one member of the sunflower $\mathcal{F}_{m+1}$. Hence there is a member $F^{m+1}_{\ell_{m+1}}$ avoiding all forbidden vertices. Continuing gives a sunflower in $\mathcal{F}$ with kernel $K$ and $p^i$ petals.

Since $p^i\geq p^{|K|}$ and $K\neq\emptyset$, it follows that $K\in\mathcal{B}_{s,k}(\mathcal{F})$, again contradicting the minimality of the $B_j$. This proves the lemma.
\end{proof}

Erd\H{o}s and Rado conjectured that any large enough $k$-uniform family must contain a sunflower.
\begin{conj}[Sunflower conjecture \cite{erdos1960intersection}]
Let $ r \geq 3 $. There exists $ c = c(r) $ such that any $k$-uniform family $ \mathcal{F}$ of size $ |\mathcal{F}| \geq c^k $ contains an $ r $-sunflower.
\end{conj}

Note that a $k$-uniform family can be replaced by a family of sets each of size at most $k$, because we can add distinct elements to each set of size less than $k$ so that its size reaches $k$, without altering any sunflowers. Erd\H{o}s and Rado proved the following theorem known as Sunflower Lemma.
\begin{theorem}[Sunflower Lemma \cite{erdos1960intersection}]\label{sunflower lem}
Let $ r \geq 3 $ and $ \mathcal{F}\subseteq\binom{[n]}{\leq k} $ be a family of sets each of size at most k. If $ |\mathcal{F}| \geq k! \cdot (r-1)^k $. Then $ \mathcal{F} $ contains an $ r $-sunflower.
\end{theorem}
The bound in Theorem \ref{sunflower lem} was improved by Alweiss, Lovett, Wu and Zhang \cite{alweiss2020improved} after sixty years, and they proved that any $k$-uniform family of size $Cr^3\log k\log{\log k}$ must contain a sunflower. By the above Sunflower Lemma and Lemma \ref{nosunflower}, the following lemma holds.

\begin{lemma}\label{bound for sunflower base}
Let $\mathcal{F}\subseteq\binom{[n]}{k}$ with matching number $\nu(\mathcal{F})<s$. There exists a function $f(s,k)$, such that $|\mathcal{B}_{s,k}^{*}(\mathcal{F})|\leq f(s,k)$ $($i.e., the size of $\mathcal{B}_{s,k}^{*}(\mathcal{F})$ has an upper bound that is independent of $n$$)$.
\end{lemma}
\begin{proof}
Put $p=sk-k+1$. By Lemma~\ref{nosunflower}, for every $1\leq i\leq k-1$, the $(i+1)$-uniform family $\mathcal{B}_{s,k}^{*(i+1)}(\mathcal{F})$ contains no sunflower of cardinality $p^i$. The Sunflower Lemma therefore gives
$$
|\mathcal{B}_{s,k}^{*(i+1)}(\mathcal{F})|<(i+1)!\,(p^i-1)^{i+1}
\qquad(1\leq i\leq k-1).
$$
Additionally, since $\nu(\mathcal{B}_{s,k}^{*(1)}(\mathcal{F}))<s$, we have $|\mathcal{B}_{s,k}^{*(1)}(\mathcal{F})|\leq s-1$. Define
$$
f(s,k)=s+\sum_{i=1}^{k-1}(i+1)!\,(p^i-1)^{i+1}.
$$
Then $|\mathcal{B}_{s,k}^{*}(\mathcal{F})|\leq f(s,k)$, and this bound is independent of $n$.
\end{proof}

\begin{lemma}\label{reduction}
Let $k\geq3$ and $\mathcal{F}\subseteq\binom{[n]}{k}$ have matching number $\nu(\mathcal{F})<s$ and $\mathcal{B}_{s,k}^{*(1)}(\mathcal{F})=\emptyset$. Assume $n$ is sufficiently large relative to the fixed integers $k$ and $s$. If $d_i({\mathcal{F}})\geq \binom{n-1}{k-1}-\binom{n-s}{k-1}$ for some $i\in [n]$, then $d_i ({\mathcal{B}_{s,k}^{*(2)}(\mathcal{F})})\geq s-1$.
\end{lemma}
\begin{proof}
Suppose to the contrary, without loss of generality, that $d_1({\mathcal{F}})\geq \binom{n-1}{k-1}-\binom{n-s}{k-1}$ and $d_1 ({\mathcal{B}_{s,k}^{*(2)}(\mathcal{F})})\leq s-2$.

By Lemma \ref{upp1} and  $\mathcal{B}_{s,k}^{*(1)}(\mathcal{F})=\emptyset$, we have
\begin{align*}
    d_1({\mathcal{F}}) &\leq \sum^k_{i=2}|\mathcal{B}_{s,k}^{*(i)}(\mathcal{F},1)|\binom{n-i}{k-i}+\sum^k_{i=2}|\mathcal{B}_{s,k}^{*(i)}(\mathcal{F})\setminus \mathcal{B}_{s,k}^{*(i)}(\mathcal{F},1)|\binom{n-i-1}{k-i-1}\\
    &\leq d_1 ({\mathcal{B}_{s,k}^{*(2)}(\mathcal{F})})\binom{n-2}{k-2}+\sum^k_{i=3}|\mathcal{B}_{s,k}^{*(i)}(\mathcal{F},1)|\binom{n-i}{k-i}\\
    &\qquad +\sum^k_{i=2}|\mathcal{B}_{s,k}^{*(i)}(\mathcal{F})\setminus \mathcal{B}_{s,k}^{*(i)}(\mathcal{F},1)|\binom{n-i-1}{k-i-1}\\
    &\leq d_1 ({\mathcal{B}_{s,k}^{*(2)}(\mathcal{F})})\binom{n-2}{k-2} + \binom{n-3}{k-3}\Bigl(\sum^k_{i=3}|\mathcal{B}_{s,k}^{*(i)}(\mathcal{F},1)|+\sum^k_{i=2}|\mathcal{B}_{s,k}^{*(i)}(\mathcal{F})\setminus \mathcal{B}_{s,k}^{*(i)}(\mathcal{F},1)|\Bigr)\\
    &\leq (s-2)\binom{n-2}{k-2}+f(s,k)\binom{n-3}{k-3}\\
    &= (s-2+o(1))\binom{n-2}{k-2}.
\end{align*}
On the other hand,
$$
d_1({\mathcal{F}})\geq \binom{n-1}{k-1}-\binom{n-s}{k-1} = (s-1+o(1))\binom{n-2}{k-2}.
$$
Here $o(1)$ is taken as $n\to\infty$ with $k$ and $s$ fixed. For sufficiently large $n$, these two estimates yield a contradiction. Therefore, $d_1 ({\mathcal{B}_{s,k}^{*(2)}(\mathcal{F})})\geq s-1$, as desired.
\end{proof}

\subsection{The $(k+2s-2)$-th largest degree}

A \emph{vertex cover} of a (hyper)graph $ G = (V, E) $ is a subset of vertices $ S \subseteq V $ such that every edge of the graph is incident to at least one vertex in $ S $.  The \emph{vertex cover number} of $G$, denoted by $ \tau(G) $, is the cardinality of a smallest vertex cover in $ G $. Formally,
\[
\tau(G) = \min\{\, |S| : S \text{ is a vertex cover of } G \,\}.
\]
A vertex cover is also known as a \emph{transversal} or a \emph{Covering set}.

Following the discussion in the previous subsection, we transform the study of $\mathcal{F}$ into the study of $\mathcal{B}_{s,k}^{*(2)}(\mathcal{F})$. Note that $\mathcal{B}_{s,k}^{*(2)}(\mathcal{F})$ is a simple graph; therefore, we need some results on simple graphs regarding degrees and matching numbers. More specifically, we present the following two lemmas.

\begin{lemma}\label{str1}
Let $s$ be a positive integer and  $G$ be a simple graph on $n$ vertices.
If there exist $2s$ vertices in $G$ each of degree at least $s$, then
 $\nu(G)\geq s$.
\end{lemma}

\begin{proof}
Assume, for the contradiction, that $\nu(G) < s$.
Let $M=\{a_1b_1,a_2b_2,\ldots,a_tb_t\}$ be a maximum matching of $G$. Then $t \le s-1$.
Denote by $U=\{a_i, b_i:i\in[t]\}$ the set of vertices covered by $M$. Then $|U| = 2t \le 2s-2$.
Let $W = V(G) \setminus U$ be the set of vertices not covered by $M$.

Let $S$ be the set of vertices of degree at least $s$. Then $|S| \ge 2s$.
Since $|U| \le 2s-2$,  $|W \cap S| \ge 2$. Pick two distinct vertices $x,y \in W \cap S$. Then $|N(x)|,|N(y)| \ge s$, where $N(x)=\{z\in V(G):xz\in E(G)\}$.

Let $X = N(x)$. Define  $X^* = \{a_i : b_i \in X,i\in [t]\} \cup \{b_i : a_i \in X,i\in [t]\}$; equivalently, $X^*$ is the set of vertices matched to $X$ under $M$. Notice that $|X^*| = |X| \ge s$. Since $M$ is a maximum matching of $G$, $G$ contains no $M$-augmenting path which implies $W$ is an independent and $N(y) \cap X^* = \emptyset$. Thus $N(y) \subseteq U \setminus X^*$.
Now we estimate the size of $N(y)$. We have
\begin{align*}
|N(y)| &\leq |U \setminus X^*| \leq |U| - |X^*| \\
       &\leq |U| - |X| \leq (2s-2) - s = s-2,
\end{align*}
which contradicts the fact that $|N(y)| \ge s$.
\end{proof}

\begin{lemma}\label{str2}
Let $s$ be a positive integer and  $G$ be a simple graph on $n$ vertices.
If $\nu(G) < s$ and $G$ contains at least $2s+1$ vertices each of degree at least $s-1$, then  $\tau(G) \le s-1$.
\end{lemma}

\begin{proof}
Let $M=\{a_1b_1,a_2b_2,\dots,a_tb_t\}$ be a maximum matching of $G$. Then $t=\nu(G)\leq s-1$. Let $U=\{a_i,b_i:i\in[t]\}$, let $W=V(G)\setminus U$, and let $S=\{v\in V(G):\deg(v)\geq s-1\}$. Since $|S|\geq2s+1$ and $|U|\leq2s-2$, we have $|W\cap S|\geq3$. The set $W$ is independent.

Pick $x\in W\cap S$ and set $X=N(x)$. Let $X^*$ be the set of vertices matched by $M$ to the vertices of $X$. Then $X\subseteq U$ and $|X^*|=|X|$. For every $u\in(W\cap S)\setminus\{x\}$, the absence of an $M$-augmenting path of length three gives $N(u)\cap X^*=\emptyset$. Hence
$$
s-1\leq |N(u)|\leq |U|-|X^*|\leq(2s-2)-(s-1)=s-1.
$$
All inequalities are equalities. Thus $t=s-1$, $|X|=|X^*|=s-1$, and
$$
N(u)=U\setminus X^*\qquad\text{for every }u\in(W\cap S)\setminus\{x\}.
$$

We claim that $X\cap X^*=\emptyset$. If an edge of $M$ had both endpoints in $X$, then, because $|M|=|X|$, another edge of $M$ would have both endpoints outside $X$, equivalently both endpoints in $U\setminus X^*$. Two distinct vertices $u_1,u_2\in(W\cap S)\setminus\{x\}$, together with this latter matching edge, would form an $M$-augmenting path, a contradiction. Therefore
$$
U=X\mathbin{\dot\cup}X^*,
\qquad N(u)=X\quad\text{for every }u\in W\cap S.
$$

\medskip
\noindent\textbf{Claim 1.} $X$ is a vertex cover of $G$.

\noindent\textbf{Proof of Claim 1.} Suppose that $uv\in E(G)$ and $\{u,v\}\cap X=\emptyset$. Since $W$ is independent, we may assume $u\in U\setminus X=X^*$, and let $u^*\in X$ be its mate in $M$.

If $v\in X^*$, let $v^*\in X$ be its mate and choose two distinct vertices $x_1,x_2\in W\cap S$. Then
$$
(M\setminus\{uu^*,vv^*\})\cup\{uv,u^*x_1,v^*x_2\}
$$
is a matching of size $t+1$, a contradiction. If $v\in W$, then $v\notin W\cap S$ because every vertex in $W\cap S$ has neighborhood $X$. Choose $x_1\in W\cap S$ distinct from $v$. Then
$$
(M\setminus\{uu^*\})\cup\{uv,u^*x_1\}
$$
is a matching of size $t+1$, again a contradiction. This proves the claim. \hfill\qed

By Claim 1 and $|X|=s-1$, we have $\tau(G)\leq s-1$.
\end{proof}
The following lemma is very crucial for the proof of our theorem.
\begin{lemma}\label{cover}
Let $s\geq 2,k\geq 3$ and $\mathcal{F}\subseteq\binom{[n]}{k}$ with matching number $\nu(\mathcal{F})<s$ and $\mathcal{B}_{s,k}^{* (1)}(\mathcal{F})=\emptyset$. Assume $n$ is sufficiently large relative to $k$ and $s$. If $d_i({\mathcal{F}})\geq \binom{n-1}{k-1}-\binom{n-s}{k-1}$ for every $i\in [k+2s-2]$, then $\tau(\mathcal{F})\le s-1$.
\end{lemma}
\begin{proof}
By Lemmas \ref{mat for 2} and  \ref{reduction}, we have
\begin{align*}
    \nu({\mathcal{B}_{s,k}^{*(2)}(\mathcal{F})})\leq&\nu({\mathcal{B}_{s,k}^{*}(\mathcal{F})})<s,\\ \text{and }d_i({\mathcal{B}_{s,k}^{*(2)}(\mathcal{F})})\geq& s-1,\text{ for every }i\in [k+2s-2].
\end{align*}
Viewing ${\mathcal{B}_{s,k}^{*(2)}(\mathcal{F})}$ as a simple graph and noting that $k+2s-2\geq 2s+1$, Lemma \ref{str2} implies that ${\mathcal{B}_{s,k}^{*(2)}(\mathcal{F})}$ has a vertex cover of size $s-1$.
Without loss of generality, we may assume that $[s-1]$ is such a vertex cover.
Since $d_i({\mathcal{B}_{s,k}^{*(2)}(\mathcal{F})})\geq s-1$ for $i\in [s,k+2s-2]$ and $[s-1]$ is a vertex cover, we obtain
$$
K_{s-1,k+s-1}\cong G:=\bigl\{\{i,j\}:1\leq i\leq s-1 < j\leq k+2s-2\bigr\}\subseteq \mathcal{B}_{s,k}^{*(2)}(\mathcal{F}).
$$

Now we claim that $[s-1]$ is also a vertex cover of $\mathcal{F}$.

Suppose to the contrary that there exists $F\in \mathcal{F}$ such that  $F\cap[s-1]=\emptyset$.
Choose $B\in \mathcal{B}_{s,k}^{*}(\mathcal{F})$ such that $B\subseteq F$. Since $\mathcal{B}_{s,k}^{*(1)}(\mathcal{F})=\emptyset$, we have $2\leq |B|\leq k$, and $B\cap[s-1]=\emptyset$.
The right class of $G$ has $k+s-1$ vertices, and deleting the at most $k$ vertices of $B$ leaves at least $s-1$ of them. Hence $G-B:=\{e\in E(G):e\cap B=\emptyset\}$ contains a complete bipartite graph $K_{s-1,s-1}$, and therefore a matching of size $s-1$ in ${\mathcal{B}_{s,k}^{*(2)}(\mathcal{F})}$.
Together with $B$, this yields a matching of size $s$ in ${\mathcal{B}_{s,k}^{*}(\mathcal{F})}$, a contradiction with $\nu({\mathcal{B}_{s,k}^{*}(\mathcal{F})})<s$.
\end{proof}
Now we can begin to prove  Theorem \ref{thm:main2}.

\noindent\textbf{Proof of Theorem \ref{thm:main2}:} We proceed by induction on $s$.

When $s=2$, the result holds by Theorem \ref{s=2}. Now assume the statement holds for $s-1$.

\noindent\textbf{Case 1:} $k\geq 3$. Choose $n_0(k,s)$ so large that, whenever $n\geq n_0(k,s)$, we have $n-1\geq n_0(k,s-1)$ and Lemmas~\ref{reduction} and~\ref{cover} are applicable for the parameter $s$. Assume $n\geq n_0(k,s)$.

\noindent\textbf{Subcase 1:} $\mathcal{B}_{s,k}^{*(1)}(\mathcal{F})\neq\emptyset$.

Assume $\{i\}\in \mathcal{B}_{s,k}^{*(1)}(\mathcal{F})$. Consider $\mathcal{G}=\{G\in\binom{[n]\setminus\{i\}}{k}:G\in\mathcal{F}\}$. We have $\nu(\mathcal{G})<s-1$; otherwise there is a matching of size $s-1$ in $\mathcal{G}$. Together with $\{i\}$, we obtain a matching of size $s$ in $\mathcal{B}_{s,k}(\mathcal{F})$. Thus, $\nu(\mathcal{B}_{s,k}^{*}(\mathcal{F}))=\nu(\mathcal{B}_{s,k}(\mathcal{F}))\geq s$, yielding a contradiction to Lemma \ref{mat for 2}.

Now we apply the induction hypothesis to $\mathcal{G}$. We have the $(k+2s-4)$-th largest degree in $\mathcal{G}$ is no more than $\binom{n-2}{k-1}-\binom{n-s}{k-1}$. Note that $|[k+2s-3]\setminus\{i\}|\geq k+2s-4$. Then there exists $i_0\in [1,k+2s-3]\setminus\{i\}$ such that $d_{i_0}(\mathcal{G})\leq \binom{n-2}{k-1}-\binom{n-s}{k-1}$. Hence,
\begin{align*}
    d_{k+2s-2}(\mathcal{F}) &\leq d_{k+2s-3}(\mathcal{F}) \leq d_{i_0}(\mathcal{F})\\
    &\leq d_{i_0}(\mathcal{G}) + \bigl|\{F\in\binom{[n]}{k}: i\in F,\ i_0\in F\}\bigr|\\
    &\leq \left(\binom{n-2}{k-1}-\binom{n-s}{k-1}\right) + \binom{n-2}{k-2}\\
    &= \binom{n-1}{k-1}-\binom{n-s}{k-1}.
\end{align*}

\noindent\textbf{Subcase 2:} $\mathcal{B}_{s,k}^{*(1)}(\mathcal{F})=\emptyset$.

Assume $d_{1}(\mathcal{F})\geq d_{2}(\mathcal{F})\geq\cdots\geq d_{k+2s-2}(\mathcal{F})\geq \binom{n-1}{k-1}-\binom{n-s}{k-1}$; otherwise, our result holds. By Lemma \ref{cover}, $\mathcal{F}$ has a cover of size at most $s-1$; enlarge it if necessary to a cover $S$ with $|S|=s-1$. For any $i\in [n]\setminus S$,
\begin{align*}
    d_i(\mathcal{F}) &\leq \bigl|\{F\in\binom{[n]}{k}: i\in F,\ F\cap S\neq\emptyset\}\bigr|\\
    &= \binom{n-1}{k-1}-\binom{n-s}{k-1}.
\end{align*}
Consequently, at most $s-1$ vertices can have degree exceeding $\binom{n-1}{k-1}-\binom{n-s}{k-1}$. Thus,
$$
d_{k+2s-2}(\mathcal{F})\leq d_s(\mathcal{F})\leq \binom{n-1}{k-1}-\binom{n-s}{k-1}.
$$

\noindent\textbf{Case 2:} $k=2$.

On the contrary, assume $d_1(\mathcal{F})\ge d_2(\mathcal{F})\ge\cdots\ge d_{2s}(\mathcal{F})\ge s$. By Lemma \ref{str1}, $\nu(\mathcal{F})\ge s$, a contradiction. \hfill\qed

\section{Remarks}\label{open}
Recall that Theorem~\ref{thm:main2} improves the index from $2sk$ to $k+2s-2$ at the cost of a weaker condition on $n$: there exists some $n_0$ (depending on $k,s$) such that for all $n\ge n_0$ and for $k\ge 2$, $s\ge 2$,
$
d_{k+2s-2}(\mathcal{F}) \le \binom{n-1}{k-1}-\binom{n-s}{k-1}.
$ The natural questions are whether the lower bound on $n$ can be made smaller or even optimal. And we give the following problem.

\begin{pro}
Let $k\ge 2$, $s\ge 2$ be integers and $\mathcal{F}$ be a $k$-uniform hypergraph on $n$ vertices with $\nu(\mathcal{F})< s$ and degree sequence $d_1(\mathcal{F})\ge d_2(\mathcal{F})\ge\cdots\ge d_n(\mathcal{F})$. Determining a suitable explicit expression of the integer $n_0$ $($expected to be linear in $k)$ such that if $n\geq n_0$,
$$
d_{k+2s-2}(\mathcal{F}) \leq \binom{n-1}{k-1}-\binom{n-s}{k-1}.
$$
\end{pro}

\vskip.2cm
\section*{Acknowledgement}
M. Cao is supported by the National Natural Science Foundation of China (Grant 12301431) and Beijing Natural Science Foundation (Grant 1262010), M. Lu is supported by the National Natural Science Foundation of China (Grant 12571372) and Beijing Natural Science Foundation (Grant 1262010).


\end{document}